\documentclass[11pt]{amsart}
\usepackage{amssymb,amscd,amsmath,enumerate,hyperref}
\usepackage[all]{xy}

\usepackage[top=30mm,right=30mm,bottom=30mm,left=30mm]{geometry}

\newtheorem{thm}{Theorem}
\newtheorem{theor}{Theorem}[section]

\newtheorem{cor}[thm]{Corollary}

\newtheorem{prop}[theor]{Proposition}

\newtheorem{lem}[theor]{Lemma}
\theoremstyle{remark}

\newtheorem{exam}[theor]{Example}
\newtheorem*{question}{Question}

\theoremstyle{definition}

\numberwithin{equation}{section}

\renewcommand{\bar}{\overline}

\newcommand{\Irr}{\mathrm{Irr}}

\newcommand{\eps}{\delta}

\newcommand{\R}{{\mathbb{R}}}

\newcommand{\F}{{\mathbb{F}}}
\newcommand{\N}{{\mathbb{N}}}

\newcommand{\Z}{{\mathbb{Z}}}
\newcommand{\ZB}{{\mathbf{Z}}}

\newcommand{\Ker}{\mathrm{Ker}}

\newcommand{\GL}{\mathrm{GL}}
\newcommand{\SL}{\mathrm{SL}}
\newcommand{\SU}{\mathrm{SU}}
\newcommand{\PSL}{\mathrm{PSL}}
\newcommand{\PSU}{\mathrm{PSU}}
\newcommand{\PGL}{\mathrm{PGL}}

\newcommand{\Sp}{\mathrm{Sp}}

\newcommand{\SSS}{{\sf S}}
\newcommand{\St}{{\sf St}}
\newcommand{\cl}{\mathfrak{l}}
\newcommand{\Alt}{{\raise 2pt\hbox{$\scriptstyle\bigwedge$}}}

\newcommand{\e}{\delta}
\newcommand{\bsig}{{\boldsymbol{\sigma}}}

\begin{document}
\title{Representations and tensor product growth}

\author{Michael Larsen}
\email{mjlarsen@indiana.edu}
\address{Department of Mathematics\\
    Indiana University \\
    Bloomington, IN 47405\\
    U.S.A.}

\author{Aner Shalev}
\email{shalev@math.huji.ac.il}
\address{Einstein Institute of Mathematics\\
    Hebrew University \\
    Givat Ram, Jerusalem 91904\\
    Israel}

\author{Pham Huu Tiep}
\email{tiep@math.rutgers.edu}
\address{Department of Mathematics\\
    Rutgers University \\
    Piscataway, NJ 08854-8019 \\
    U.S.A.}

\begin{abstract}
The deep theory of approximate subgroups establishes 3-step product growth for subsets of finite simple groups $G$ of Lie type of bounded rank.
In this paper we obtain 2-step growth results for representations of such groups $G$ (including those of unbounded rank), where products of
subsets are replaced by tensor products of representations.

Let $G$ be a finite simple group of Lie type and $\chi$ a character of $G$.  Let $|\chi|$ denote the sum of the squares of the degrees of all (distinct) irreducible characters of $G$ which are constituents of $\chi$.  We show that for all $\delta>0$ there exists $\epsilon>0$, independent of $G$, such that if $\chi$ is an irreducible character of $G$ satisfying $|\chi| \le |G|^{1-\delta}$, then $|\chi^2| \ge |\chi|^{1+\epsilon}$. We also obtain results for reducible characters, and establish faster growth in the case where $|\chi| \le |G|^{\delta}$.

In another direction, we explore covering phenomena, namely situations where every irreducible character of $G$ occurs as a
constituent of certain products of characters. For example, we prove that if $|\chi_1| \cdots |\chi_m|$ is a high enough power of $|G|$,
then every irreducible character of $G$ appears in $\chi_1\cdots\chi_m$. Finally, we obtain growth results for compact semisimple Lie groups.
\end{abstract}

%\subjclass{Primary 20G40; Secondary 20D06, 11P05}

\thanks{ML was partially supported by the NSF grant DMS-2001349.
AS was partially supported by ISF grant 686/17 and the Vinik Chair
of mathematics.
PT was partially supported by the NSF grant
DMS-1840702, the Joshua Barlaz Chair in Mathematics, and the Charles Simonyi Endowment at the
Institute for Advanced Study (Princeton). All three authors were partially supported by BSF grant 2016072.}

\maketitle

\section{Introduction}

In the past two decades there has been intense interest in growth phenomena in groups and in
finite simple groups in particular; see \cite{H1,Hr,BGT,H2,PS,B}. The celebrated Product Theorem, proved independently in \cite{BGT} and in \cite{PS}, shows that, if $G$ is a finite simple
group of Lie type of rank $r$, and $A \subset G$ is a generating set, then either $A^3 = G$ or
$|A^3| \ge |A|^{1 + \epsilon}$ where $\epsilon > 0$ depends only on $r$. Here $A^k$ denotes the set of all products
$a_1 a_2 \cdots a_k$ where $a_1, \ldots , a_k \in A$.

The main goal of this paper is to study analogous power growth phenomena in representation theory,
with emphasis on (complex) representations of finite simple groups $G$ of Lie type. Here products of subsets of $G$ are replaced
by tensor products of representations. Our results on tensor product growth are somewhat stronger than the Product Theorem in two senses:
we establish 2-step (instead of 3-step) growth, as well as uniform growth when the rank of $G$ tends to infinity.

Let $G$ be any finite group.  If $X = \{\chi_1,\ldots,\chi_k\}$ is a set of (pairwise distinct) irreducible characters of $G$, we define
$$|X| = \sum_{i=1}^k \chi_i(1)^2.$$
This is the Plancherel measure, normalized so that $|\Irr(G)| = |G|$.  If $\chi$ is any character of $G$, we define $|\chi| = |X_\chi|$, where $X_\chi$ denotes the set of distinct irreducible constituents of $\chi$.  Recall that $G$ is {\it quasisimple} if
$G=[G,G]$ and $G/\mathbf{Z}(G)$ is simple.

Our first growth results are as follows:

\begin{thm}
\label{main irred}
For all $\delta > 0$, there exists $\epsilon > 0$ such that if $G$ is a finite quasisimple group of Lie type and $\chi$ is an irreducible character of $G$ with $|\chi|\le |G|^{1-\delta}$,
then $|\chi^2| \ge |\chi|^{1+\epsilon}$ and $|\chi\bar\chi|\ge |\chi|^{1+\epsilon}$.
\end{thm}

We also have a version of this result for general characters in groups of high rank:

\begin{thm}
\label{main high rank}
For all $\delta > 0$, there exist $\epsilon > 0$ and $R>0$ such that if $G$ is a finite quasisimple group of Lie type and rank $\ge R$, and $\chi$ is any (not necessarily irreducible) character of $G$ with $|\chi|\le |G|^{1-\delta}$,
then $|\chi^2| \ge |\chi|^{1+\epsilon}$ and $|\chi\bar\chi|\ge |\chi|^{1+\epsilon}$.
\end{thm}

The proofs of these results present $\epsilon$ as an explicit function of $\delta$, e.g. $\epsilon = \frac{c\delta}{4+2c(1-\delta)}$
in Theorem \ref{main irred}, where $c > 0$ is the absolute constant in \cite[Theorem A]{LT}. Moreover, if $G$ is sufficiently large
but of bounded rank $r$, and $\chi$ is irreducible, then  $\epsilon = \frac\delta{2-2\delta}$ will do; for example,
any irreducible character $\chi$ of $G$ with $|\chi|\le |G|^{1/2}$ satisfies $|\chi^2| \ge |\chi|^{3/2}$.

\smallskip

Power growth of conjugacy classes $A$ of arbitrary finite simple groups $G$ was studied before the Product Theorem was
proved. It is shown in \cite[2.7]{Sh} there that for any $\delta > 0$ there exists $\epsilon > 0$, depending only on $\delta$, such that
$|A| \leq |G|^{1-\delta}$ implies $|A^3| \geq |A|^{1+ \epsilon}$.
Furthermore, if $G$ is of Lie type then $|A^2| \geq |A|^{1+\epsilon}$ where $\epsilon > 0$ depends only on the rank of $G$ \cite[10.4]{Sh}.

Subsequently, growth of general normal subsets (namely, union of conjugacy classes) was also studied. It is shown in \cite[1.5]{GPSS}
that there are absolute constants $N \in \N$ and $\epsilon > 0$ such that for any normal subset $A$ of a finite simple group $G$, either
$A^N = G$ or $|A^2| \geq |A|^{1+\epsilon}$.

In \cite{LSS} faster growth of the form $|A^2| \ge |A|^{2-\epsilon}$ for small normal subsets $A$ of arbitrary finite simple groups is established.
Our next result gives a character-theoretic analogue of \cite[Theorem 1.3]{LSS}:

\begin{thm}\label{prod}
For any $\epsilon > 0$, there exists an explicit $\delta > 0$ such that the following statement holds. If $G$ is a finite quasisimple group
of Lie type and $\chi_1,\chi_2$ are any (not necessarily irreducible) characters of $G$ with $|\chi_1|, |\chi_2| \leq |G|^\delta$, then
$$|\chi_1\chi_2| \geq \bigl(|\chi_1|\cdot|\chi_2|\bigr)^{1-\epsilon}.$$ In particular, if $\chi$ is a character of $G$ satisfying
$|\chi| \leq |G|^{\delta}$ then $|\chi^2| \geq |\chi|^{2-2\epsilon}$.
\end{thm}

We note that $|\chi_1\chi_2| \leq |\chi_1| \cdot |\chi_2|$ (see Lemma \ref{sub-add-mult} below), hence the growth
established in Theorem \ref{prod} is almost best possible. As a consequence of Theorem \ref{prod}, we obtain:

\begin{cor}\label{prod-k}
For any $\epsilon > 0$ and any integer $k \geq 2$, there exists an explicit $\gamma = \gamma(\epsilon,k) > 0$ such that the following statement holds.
If $G$ is a finite quasisimple group
of Lie type and $\chi_1,\chi_2, \ldots,\chi_k$ are any (not necessarily irreducible) characters of $G$ with
$|\chi_1|, |\chi_2|, \ldots, |\chi_k| \leq |G|^\gamma$, then
$$|\chi_1\chi_2 \cdots \chi_k| \geq \bigl(|\chi_1|\cdot|\chi_2|\cdots |\chi_k|\bigr)^{1-\epsilon}.$$
In particular, if $\chi$ is a character of $G$ satisfying
$|\chi| \leq |G|^{\gamma}$ then $|\chi^k| \geq |\chi|^{k-k\epsilon}$.
\end{cor}

The above result shows that, for any $\epsilon > 0$ and any integer $k \ge 2$ there exists an explicit $\delta = \delta(\epsilon,k) > 0$
such that, for $G$ as above and any (not necessarily irreducible) character $\chi$ of $G$
satisfying $|\chi| \leq |G|^{\delta}$ we have $|\chi^k| \geq |\chi|^{k-\epsilon}$; indeed, define $\delta(\epsilon,k) = \gamma(\epsilon/k, k)$.

Applying Theorem \ref{prod} we deduce the following result, which is a character-theoretic analogue of \cite[Theorem 1.1]{LSS}:

\begin{thm}\label{char-growth}
For all $\epsilon > 0$, there exists an explicit positive integer $N$ such that if $G$ is a finite simple group of Lie type and $\chi$ is any (not necessarily irreducible) character of $G$, then either $\chi^N$ contains every irreducible character of $G$, or
$|\chi^2| \ge |\chi|^{2-\epsilon}$.
\end{thm}

The analogy with Gowers' theorem raises the question of whether $N=3$ suffices in Theorem~\ref{char-growth}
when $|\chi|$ is sufficiently large.
A recent theorem of Sellke \cite[Theorem 1.2]{Se} shows that the answer to this question is affirmative for large $G$ if $\chi$ is so large that $\frac{|\chi|}{|G|}$ is bounded away from $0$.
We therefore ask the following:

\begin{question}
If $G$ is a finite simple group of Lie type and $\chi$ is an arbitrary character of $G$ such that $\frac{\log |\chi|}{\log |G|}$ is sufficiently close to $1$, is it true that $|\chi^3| = |G|$?
\end{question}

We remark that the example of $\PSU_{2n+1}(q)$ \cite[Theorem 1.2]{HSTZ} shows, in general, it would be too much to ask for $|\chi^2| = |G|$.  On the other hand,
for certain simple groups of Lie type, we can bring $N$ down to $6$ or $7$.

\begin{thm}
\label{N=6}
If $G = \PSL_n(q)$ and $q$ is sufficiently large in terms of $n$, then
$|\chi|\ge |G|^{11/12}$ implies $|\chi^6| = |G|$. If $G = \PSU_n(q)$ and $q$ is sufficiently large in terms of $n$, then
$|\chi|\ge |G|^{11/12}$ implies $|\chi^7| = |G|$.
\end{thm}

We also offer a character-theoretic analogue of the Rodgers-Saxl theorem on products of conjugacy classes in \cite{RS}. In the case that the characters are irreducible, this analogue was
conjectured by Gill in \cite{Gi} and proved in \cite[Theorem 8.5]{LT}.

\begin{thm}\label{cover}
There exists an explicit constant $c > 0$ such that the following statement holds. If $G$ is a finite simple group
of Lie type, $m \geq 1$ any integer, and $\chi_1,\chi_2, \ldots, \chi_m$ are any (not necessarily irreducible) characters of $G$ with
$\prod^m_{i=1}|\chi_i| \geq |G|^c$, then $|\chi_1\chi_2 \ldots \chi_m| = |G|$ and thus $\chi_1\chi_2 \ldots \chi_m$ contains every
irreducible character of $G$.
\end{thm}

Finally, we prove an analogue of Theorem~\ref{main irred} for compact semisimple Lie groups.

\begin{thm}\label{compact}
Let $G$ be a compact semisimple Lie group.  Then there exists $\epsilon > 0$ such that for each irreducible character $\chi$ of $G$, we have $|\chi^2| \ge |\chi|^{1+\epsilon}$.
\end{thm}

It is possible that some of the presented results also hold for alternating groups (and likewise, Theorem \ref{main high rank} may hold
for Lie-type groups of bounded rank as well). However, the techniques developed in the paper do not seem to apply to these open cases.

\smallskip
Some words on the structure of this paper. In Section 2 we prove some preliminary results for finite groups.
Section 3 is devoted to the proofs of Theorem~\ref{main irred} and Theorem~\ref{main high rank}.
In Section 4 we prove Theorems~\ref{prod}, ~\ref{char-growth}, ~\ref{cover}, and Corollary ~\ref{prod-k},
while the proof of Theorem~\ref{N=6} is carried out in Section 5.
In Section 6 we study tensor product growth of representations of semisimple compact Lie groups and prove Theorem~\ref{compact}.

\section{Preliminaries}
We begin with general inequalities for finite groups $G$.

\begin{lem}\label{sub-add-mult}
Let $\alpha$, $\beta$ be any (not necessarily irreducible) characters of $G$. Then
\begin{enumerate}[\rm(a)]
\item $|\alpha+\beta| \leq |\alpha|+|\beta|$.
\item $|\alpha\beta| \leq |\alpha| \cdot |\beta|$.
\end{enumerate}
\end{lem}

\begin{proof}
Without any loss of generality we may assume that $\alpha = \sum_i\alpha_i$ and $\beta=\sum_j \beta_j$
are mutliplicity-free, with $\alpha_i,\beta_j \in \Irr(G)$.
Then
$$|\alpha+\beta| = |\sum_i\alpha_i+\sum_j\beta_j| \leq \sum_i\alpha_i(1)^2+\sum_j\beta_j(1)^2 = |\alpha|+|\beta|,$$
proving (a).
Next, $|\alpha_i\beta_j| \leq \bigl(\alpha_i\beta_j(1)\bigr)^2 = \alpha_i(1)^2\beta_j(1)^2$. It follows from
(a) that
$$|\alpha\beta| = |\sum_{i,j}\alpha_i\beta_j| \leq \sum_{i,j}|\alpha_i\beta_j| \leq \sum_{i,j}\alpha_i(1)^2\beta_j(1)^2
   = \bigl(\sum_i\alpha_i(1)^2 \bigr) \cdot \bigl(\sum_j \beta_j(1)^2 \bigr) = |\alpha| \cdot |\beta|,$$
proving (b).
\end{proof}

\begin{prop}\label{lem:bound1}
Let $\chi_1, \ldots,\chi_n$ be irreducible characters of $G$, $n\ge 2$.
\begin{enumerate}[\rm(i)]
\item In general, we have
$$|\chi_1 \ldots \chi_n| \geq \max_{1 \leq i \leq n} |\chi_i|.$$
\item If $G$ is quasisimple and $\chi_i \neq 1_G$ for $1 \leq i \leq n$, then
$$|\chi_1 \ldots \chi_n| > \max_{1\leq i\leq n}|\chi_i|.$$
\item If $G$ is perfect, $\chi_1\neq 1_G$, and $\chi_i\in \{\chi_1,\bar\chi_1\}$ for all $i$, then $|\chi_1\cdots\chi_n| > |\chi_1| = \cdots = |\chi_n|$.
\end{enumerate}
\end{prop}

\begin{proof}
(i) Without any loss of generality, assume $\chi_1(1) = \max_i\chi_i(1)$. Note that
\begin{equation*}
%\label{square of product}
\begin{split}
\langle \chi_1 \ldots \chi_n,\chi_1 \ldots \chi_n\rangle &=\frac 1{|G|} \sum_{g\in G} |\chi_1(g) \ldots \chi_n(g)|^2  \\
&\le\frac 1{|G|} \max_{h\in G} |\chi_2(h) \ldots \chi_n(h)|^2 \sum_{g\in G} |\chi_1(g)|^2 \\
&= \max_{h\in G} |\chi_2(h) \ldots \chi_n(h)|^2 \langle \chi_1,\chi_1\rangle \\
&= \chi_2(1)^2 \ldots \chi_n(1)^2,
\end{split}
\end{equation*}
with equality only if, for every $g \in G$,
\begin{equation}\label{eq:bd1}
  \mbox{either }\chi_1(g)=0 \mbox{ or  }|\chi_2(g)\ldots \chi_n(g)| = \chi_2(1) \ldots \chi_n(1).
\end{equation}

If $d_1,d_2,\ldots,d_k$ are the degrees of the distinct irreducible constituents of $\chi_1 \ldots \chi_n$ and $m_1,\ldots,m_k$ their multiplicities in $\chi_1 \ldots \chi_n$,
then
$$\sum_i m_i^2 = \langle \chi_1 \ldots \chi_n,\chi_1 \ldots \chi_n\rangle \le \chi_2(1)^2 \ldots\chi_n(1)^2$$ and
$$\sum_i d_i m_i = \chi_1(1) \ldots \chi_n(1).$$
So by the Cauchy-Schwarz inequality,
$$|\chi_1 \ldots \chi_n| = \sum^k_{i=1} d_i^2 \ge \frac{\chi_1(1)^2 \ldots \chi_n(1)^2}{\sum^k_{i=1} m_i^2} \ge \chi_1(1)^2 = |\chi_1|,$$
as stated.

\smallskip
(ii) Suppose $G$ is quasisimple, $\chi_i \neq 1_G$ for all $i$, but $|\chi_1 \ldots \chi_n|=|\chi_1|=\max_i|\chi_i|$. Note that $K_i:= \{ g \in G \mid |\chi_i(g)| = \chi_i(1)\}$ is a normal subgroup of $G$, which is contained in $\ZB(G)$. Setting $K:= \cap^n_{i=2}K_i$ and
using \eqref{eq:bd1}, we see that $\chi_1(g)=0$ for all
$g \notin K$. As $K \leq \ZB(G)$, we have $|\chi_1(g)| = \chi_1(1)$ for $g \in K$.
It follows that
$$|G| = \sum_{g \in G}|\chi_1(g)|^2 = |K|\chi_1(1)^2,$$
and thus $\chi_1(1) = |G/K|^{1/2}$, i.e. $\chi_1$ is of central type character for $G/\Ker(\chi)$.
By the Howlett-Isaacs
theorem \cite{HI}, $G/\Ker(\chi_1)$ is solvable, which is impossible, since it maps onto the simple group $G/\ZB(G)$.

\smallskip
(iii) Suppose $G$ is perfect and $\chi_1 \neq 1$, but $|\chi_1^{n-j} \bar\chi_1^j|=|\chi_1|$ for some $0 \leq j \leq n$. We again have
that $K:= \{ g \in G \mid |\chi_1(g)| = \chi_1(1)\}$ is a normal subgroup of $G$, and $\chi_1(g)=0$ for all
$g \notin K$ by \eqref{eq:bd1}. It follows that
$$|G| = \sum_{g \in G}|\chi_1(g)|^2 = |K|\chi_1(1)^2,$$
and thus $\chi_1(1) = |G/K|^{1/2}$, i.e. $\chi_1$ is of central type character for $G/\Ker(\chi_1)$.
By the Howlett-Isaacs
theorem, $G/\Ker(\chi_1)$ is solvable, so $\Ker(\chi_1) \geq [G,G]=G$ and thus $\chi_1=1_G$, a contradiction.
\end{proof}

%Note that if $G$ is a finite simple group and $\chi\neq 1$, then $|\chi(g)| < \chi(1)$ for all $g\neq 1$, so the inequalities above are strict.
%Therefore,

The following example shows the absence of growth in $G$ in a class of groups $G$ which are far from being simple.

\begin{exam}
Let $p$ be any prime, $n \geq 1$ any integer, and $G$ be an extraspecial $p$-group of order $p^{1+2n}$. Then $G$ has $p^{2n}$ linear
characters and $p-1$
irreducible characters of degree $p^n$; let $\chi$ be one of the latter. Now, if $k \geq 1$ be any integer, then
$\chi^k$ is the sum of $p^{2n}$ linear characters if $p|k$, and a multiple of a single irreducible character of degree $p^n$ if $p \nmid k$.
Thus we always have $|\chi^k| = p^{2n} = |\chi| = |G|^{1-1/(2n+1)}$.
\end{exam}

\begin{question}
When $\chi_i$ are not assumed irreducible, can it ever happen that
$$|\chi_1\cdots\chi_n| < \max(|\chi_1|,\ldots,|\chi_n|)?$$
\end{question}

\section{Proofs of Theorems \ref{main irred} and \ref{main high rank}}

We begin this section by proving Theorem~\ref{main high rank} in the case that $\chi$ is irreducible, which is also the high rank case of Theorem~\ref{main irred}.

\begin{prop}
\label{high rank irred}
For all $\delta > 0$, there exist $\epsilon > 0$ and $R>0$ such that if $G$ is a finite quasisimple group of Lie type and rank $\ge R$, and $\chi\in\Irr(G)$ satisfies $|\chi|\le |G|^{1-\delta}$,
then $|\chi^2| \ge |\chi|^{1+\epsilon}$ and $|\chi\bar\chi|\ge |\chi|^{1+\epsilon}$.

\end{prop}

\begin{proof}
Let $r\ge R$ be the rank of $G$ and $q$ the cardinality of the field of definition (for Suzuki and Ree groups, $q$ is defined, as usual, so that $q^2$ is an odd power of $2$ or $3$.)
By the main theorem of \cite{LT}, there exists $c>0$ such that
$$|\chi(g)| \le \chi(1)^{1-c\frac{\log |g^G|}{\log |G|}}.$$
Given $\alpha > 0$, let $S=S_\alpha\subset G$ denote the set of elements such that $|g^G|\le |G|^{\alpha / 2}$.  Then, by a well known result of Fulman and Guralnick \cite{FG}, we have
$$|S| \le k(G) |G|^{\alpha / 2} \le 27.2 q^r|G|^{\alpha / 2} \le |G|^\alpha,$$
if $R$ is taken sufficiently large in terms of $\alpha$.
Thus,
\begin{align*}
\langle \chi^2,\chi^2\rangle = \langle \chi\bar\chi,\chi\bar\chi\rangle &=\frac 1{|G|} \sum_{g\in G} |\chi(g)|^4  \\
&=\frac 1{|G|}\sum_{g\in S} |\chi(g)|^4 +\frac 1{|G|}\sum_{g\in G\setminus S} |\chi(g)|^4 \\
&\le \frac{|S|}{|G|} \chi(1)^4 + \frac 1{|G|}\max_{g\in G\setminus S} |\chi(g)|^2 \sum_{g\in G\setminus S} |\chi(g)|^2\\
&\le |G|^{\alpha-1} \chi(1)^4 + \frac 1{|G|}\chi(1)^{2(1-c\alpha/2)} \sum_{g\in G\setminus S} |\chi(g)|^2\\
&\le |G|^{\alpha-1}\chi(1)^4 + \chi(1)^{2-c\alpha}.
\end{align*}
By hypothesis, $|\chi| \le |G|^{1-\delta}$, so $\chi(1)\le |G|^{1/2-\delta/2}$, which implies
$$\langle \chi^2,\chi^2\rangle \le \chi(1)^{4 - \frac{2-2\alpha}{1-\delta}} + \chi(1)^{2-c\alpha} = (\chi(1)^{-\frac{2\delta-2\alpha}{1-\delta}}+\chi(1)^{-c\alpha})\chi(1)^2.$$
Choosing
$$\alpha := \frac{2\delta}{2+c(1-\delta)},\ \epsilon := \frac{c\alpha}4,$$
we get
$$\langle \chi^2,\chi^2\rangle \le 2\chi(1)^{-c\alpha} \chi(1)^2 \le \chi(1)^{2-2\epsilon}$$
if $R$ is taken sufficiently large.
Applying the Cauchy-Schwarz inequality, we get
$$|\chi^2| \ge \frac{\chi(1)^4}{\langle \chi^2,\chi^2\rangle} \ge \chi(1)^{2+2\epsilon} = |\chi|^{1+\epsilon}.$$
Likewise,
$$|\chi\bar\chi| \ge \frac{\chi(1)^4}{\langle \chi\bar\chi,\chi\bar\chi\rangle} \ge \chi(1)^{2+2\epsilon} = |\chi|^{1+\epsilon}.$$
\end{proof}

To complete the proof of Theorem~\ref{main irred}, we need only treat the bounded rank case. Using Proposition~\ref{lem:bound1}(iii),
we may take $|G|$ to be as large as we wish.

\begin{prop}
Let $R$ be fixed.  For all $\delta > 0$, there exists $\epsilon > 0$ such that if $G$ is a quasisimple group of Lie type of rank $r < R$ and $\chi$ is an irreducible character of $G$ with $|\chi| \le |G|^{1-\delta}$, then $|\chi^2|$ and $|\chi\bar\chi|$ are both at least $|\chi|^{1+\epsilon}$.
\end{prop}

\begin{proof}
We have already remarked that $G$ can be assumed arbitrarily large.  Since the rank $r$ is bounded, this means we may take $q$ arbitrarily large.
By an estimate of Gluck \cite{Gl}, for $g\neq 1$ we have $|\chi(g)| \le C/\sqrt q$, where $C$ is an absolute constant.
On the other hand, our upper bound on the rank gives an upper bound on $|G|$ of the form $q^N$ for some $N$.
Taking $q$ sufficiently large, we therefore have
$$\frac{|\chi(g)|}{\chi(1)} \le C q^{-1/2} \le q^{-1/3} \le |G|^{-\frac 1{3N}} \le \chi(1)^{-\frac 1{3N}}$$
for all $g\neq 1$.

Setting $S = \{1\}$, we proceed as in the proof of Proposition~\ref{high rank irred}, obtaining
\begin{align*}
\langle \chi^2,\chi^2\rangle = \langle \chi\bar\chi,\chi\bar\chi\rangle &=\frac 1{|G|} \sum_{g\in G} |\chi(g)|^4  \\
&=\frac{\chi(1)^4}{|G|} +\frac 1{|G|}\sum_{g\neq 1} |\chi(g)|^4 \\
&\le \frac{\chi(1)^4}{|G|} + \frac 1{|G|}\max_{g\neq 1} |\chi(g)|^2 \sum_{g\neq 1} |\chi(g)|^2\\
&\le \chi(1)^{4-\frac 2{1-\delta}} + \chi(1)^{2(1-\frac 1{3N})}.
\end{align*}

Assuming $\delta \le \frac 1{3N+1}$, we have
$$\langle \chi^2,\chi^2\rangle \le 2\chi(1)^{2 - \frac{2\delta}{1-\delta}} \le \chi(1)^{2 - \frac\delta{1-\delta}}$$
if $|G|$, and therefore $\chi(1)$, is sufficiently large.   Taking $\epsilon := \frac\delta{2-2\delta}$ and using Cauchy-Schwarz as before, the proposition follows.
\end{proof}

Now we prove Theorem~\ref{main high rank}.

\begin{proof}
By \cite[Theorem~1.2]{LS}, for all $N$, there exists $R$ such that if $G$ has rank $r\ge R$, then
$$\sum_{\chi\in\Irr(G)} \chi(1)^{-1/N} \le \frac 32.$$
This implies that the number of irreducible characters of $G$ of degree $\le D$ is at most $D^{1/N}$.
Therefore, if $\chi$ is a character of $G$ with distinct irreducible factors $\chi_1,\ldots, \chi_k$ of degrees $d_1\le d_2\le\cdots\le d_k$, then
$$|\chi| = \sum_{i=1}^k d_i^2 \le d_k^{1/N} d_k^2,$$
so
$$|\chi_k| = d_k^2 \ge |\chi|^{\frac N{2N+1}}.$$
On the other hand, $|\chi^2| \ge |\chi_k^2|$ and $|\chi\bar\chi| \ge |\chi_k\bar\chi_k|$.
Theorem~\ref{main irred} now follows by applying Proposition~\ref{high rank irred} to $\chi_k$.
\end{proof}

\section{Proof of Theorems \ref{prod}, \ref{char-growth} and \ref{cover}}

%$\frac {\deg \sum_{\pi\vdash a} \frac {|Z(\pi)|}{a!} Q_{b\pi}^\lambda}
%        {\deg \sum_{\pi\vdash a} \frac {|Z(\pi)|}{a!} Q_{b\pi}^{1^a}}
%\le \frac{\dim C(\pi)}{\dim C(1^n)}$
The notion of the {\it level} $\cl(\chi)$ of an irreducible character $\chi \in \Irr(G)$ of a finite classical group $G$ was introduced in
\cite{GLT1} for groups of type $A$, and \cite[Definition 3.2]{GLT2} for other classical types. For brevity, we use
$\SL^\epsilon$ to denote $\SL$ when $\epsilon=+$ and $\SU$ when $\epsilon=-$.
Using the results of \cite{GLT1,GLT2},
we can prove the following bound (which is useful only when $L<\sqrt n$).

\begin{prop}\label{count}
Let $G = \SL^\epsilon_n(q)$ with $n \geq 7$, or $\Sp_{2n}(q)$, $\Omega_{2n+1}(q)$, $\Omega^{\pm}_{2n}(q)$ with $n \geq 6$,
and let $L \geq 1$ be any real number.
\begin{enumerate}[\rm(a)]
\item Suppose $G = \SL^\epsilon_n(q)$ and $L \leq n/6$. Then the number $N(L)$ of $\chi \in \Irr(G)$ with $\chi(1) \leq q^{nL}$ is at most $q^{12L^2}$.
\item Suppose $G = \Sp_{2n}(q)$ with $2 \nmid q$. Then the number $N(L)$ of $\chi \in \Irr(G)$ with $\chi(1) \leq q^{nL}$ is at most $q^{18L^2}$.
\item Suppose $G = \Sp_{2n}(q)$ with $2|q$, or $\Omega_{2n+1}(q)$ with $2 \nmid q$, or $\Omega^+_{2n}(q)$ for any $q$. Then the number $N(L)$ of $\chi \in \Irr(G)$ with $\chi(1) \leq q^{nL}$ is at most $q^{50L^2}$.
\item Suppose $G = \Omega^-_{2n}(q)$. Then the number $N(L)$ of $\chi \in \Irr(G)$ with $\chi(1) \leq q^{(n-1)L}$ is at most $q^{41L^2}$.
\end{enumerate}
\end{prop}

\begin{proof}
(a) First we bound $j:=\cl(\chi)$ in terms of $L$. If $j > n/2$, then by \cite[Theorem 1.3(ii)]{GLT1} we have
$$\chi(1) \geq (2/3)q^{n^2/4-3} > q^{n^2/6}$$
since $n \geq 7$, contradicting the condition $L \leq n/6$. Hence $j \leq n/2$. Now using \cite[Theorem 1.3(i)]{GLT1} we get
$$q^{nL} \geq \chi(1) \geq \frac{1}{2(q+1)}q^{j(n-j)} > q^{j(n-j)-2.6} \geq q^{jn/2-2.6},$$
which implies
\begin{equation}\label{eq-ct1}
  j \leq \frac{nL+2.6}{n/2} = 2L + \frac{5.2}{n} < \frac{11}{4}L.
\end{equation}

Let $\eps=+$ and consider the reducible {\it Weil character} $\tau=\tau_{n,q}$ of $\SL_n(q)$, see \cite[(1.1)]{GLT1}. Note that $\tau$ is
just the permutation character of $G=\SL_n(q)$ acting on the point set of $V:=\F_q^n$; in particular $\tau$ contains $1_G$, and so
$\tau^j$ contains all irreducible constituents of $\sum^{j}_{i=0}\tau^i$.
By definition, $j=\cl(\chi)$ is the smallest non-negative integer such that $\chi$ is a constituent of $\tau^j$.
Hence, \eqref{eq-ct1} implies that
$N(L) \leq N^*(l)$ with $l:= \lfloor 11L/4 \rfloor$, where $N^*(l)$ is the number of distinct irreducible constituents of $\tau^l$.
Next,
$$N^*(l) \leq \langle \tau^l,\tau^l \rangle = \langle \tau^{2l},1_G \rangle,$$
the number of $G$-orbits on $V^{2l}$. By
\cite[Lemma 2.4]{GLT1}, the number of $\GL_n(q)$-orbits on $V^{2l}$ is at most $8q^{l^2}$. Since $G$ has index $q-1$ in
$\GL_n(q)$, we have
$$N(L) \leq 8(q-1)q^{l^2} \leq q^{l^2+3} \leq q^{(11L/4)^2+3} < q^{11L^2}.$$

Now suppose $\eps=-$ and consider the reducible {\it Weil character} $\zeta=\zeta_{n,q}$ of $\SU_n(q)$, see \cite[(1.1)]{GLT1}. Note that
$\zeta^2=\bigl(\tau_{n,q^2}\bigr)|_G$ is
just the permutation character of $G=\SU_n(q)$ acting on the point set of $U:=\F_{q^2}^n$; in particular $\zeta^2$ contains $1_G$, and so
$\zeta^j+\zeta^{j-1}$ contains all irreducible constituents of $\sum^{j}_{i=0}\zeta^i$. Hence, \eqref{eq-ct1} implies that
$N(L) \leq N^*(l)+N^*(l-1)$ with $l:= \lfloor 11L/4 \rfloor$, where $N^*(l)$ is the number of distinct irreducible constituents of $\zeta^l$.
Next,
$$N^*(l) \leq \langle \zeta^l,\zeta^l \rangle = \langle \zeta^{2l},1_G \rangle = \langle \bigl(\tau_{n,q^2}\bigr)|_G,1_G\rangle,$$
the number of $G$-orbits on $U^{2l}$. By
\cite[Lemma 2.4]{GLT1}, the number of $\mathrm{GU}_n(q)$-orbits on $U^{2l}$ is at most $2q^{l^2}$. Since $G$ has index $q+1$ in
$\mathrm{GU}_n(q)$, we have
$$N(L) \leq 4(q+1)q^{l^2} < q^{l^2+4} \leq q^{(11L/4)^2+4} < q^{12L^2}.$$

\smallskip
(b) In the remaining cases, we set $k := \lfloor (j+2)/3 \rfloor$ for $j:=\cl(\chi)$, so that
$$j/3 \leq k \leq (j+2)/3.$$
Consider the case of
$G = \Sp_{2n}(q)$ with $2 \nmid q$. By \cite[Lemma 3.4]{GLT2}, $j \leq 2n+1$, so
$$(k+1)/2 \leq (2n+6)/6 \leq n/2$$
when $n \geq 6$. Now, applying \cite[Theorem 1.5]{GLT2} we have
$$q^{nL} \geq \chi(1) \geq q^{k(n-(k+1)/2)} \geq q^{kn/2},$$
whence $k \leq 2L$ and so
\begin{equation}\label{eq-ct2}
  j \leq 3k \leq 6L.
\end{equation}
By definition, $\cl(\chi)= j$ means that $\chi$ is an irreducible constituent of $(\omega+\omega^*)^j$, where
$\omega$ and $\omega^*$ are the two reducible {\it Weil characters} of $\Sp_{2n}(q)$, see \cite[\S3]{GLT2}. Next,
$(\omega+\omega^*)^2$ always contains $\bigl(\tau_{2n,q}\bigr)|_G$, and hence $1_G$ as well, by \cite[Proposition 3.1]{GLT2}.
It follows that $(\omega+\omega^*)^j+(\omega+\omega^*)^{j-1}$ contains all irreducible constituents of
$\sum^{j}_{i=0}(\omega+\omega^*)^i$. Hence, \eqref{eq-ct2} implies that
$N(L) \leq N^*(l)+N^*(l-1)$ with $l:= \lfloor 6L \rfloor$, where $N^*(l)$ is the number of distinct irreducible constituents of
$(\omega+\omega^*)^l$. By \cite[Proposition 3.2]{GLT2},
$\omega^2=(\omega^*)^2$, so in fact any irreducible constituent of $(\omega+\omega^*)^l$ is an irreducible constituent of
$\omega^l$ or of $\omega^{l-1}\omega^*$. If $q \equiv 3\ (\bmod\ 4)$, then by \cite[Proposition 3.2]{GLT2} for $0 \leq i \leq l$ we have
$$\langle \omega^{l-i}(\omega^*)^i, \omega^{l-i}(\omega^*)^i\rangle = \langle \omega^{l-i}\bar\omega^i, \omega^{l-i}\bar\omega^i\rangle
   = \langle (\omega\bar\omega)^{2l},1_G \rangle = \langle (\tau_{2n,q}^l)|_G,1_G \rangle,$$
the number of $G$-orbits on $V^{l}$ with $V:=\F_q^{2n}$, which is at most $6q^{l(l-1)/2}$ by \cite[Lemma 2.6]{GLT2}. It follows
that $N^*(l) \leq 12q^{l(l-1)/2}$.

If $q \equiv 1\ (\bmod\ 4)$, then by \cite[Proposition 3.2]{GLT2} for $0 \leq i \leq l$ we have
$$\langle \omega^{l-i}(\omega^*)^i, \omega^{l-i}(\omega^*)^i\rangle
   = \langle (\omega)^{2l-2i}(\omega^*)^{2i},1_G\rangle = \langle (\tau_{2n,q}^l)|_G,1_G \rangle,$$
which is again the number of $G$-orbits on $V^{l}$ with $V:=\F_q^{2n}$ and so at most $6q^{l(l-1)/2}$ by \cite[Lemma 2.6]{GLT2}. Thus
we also have $N^*(l) \leq 12q^{l(l-1)/2}$ in this case.

Thus $N(L) \leq 12q^{l(l-1)/2}+12q^{(l-1)(l-2)/2} < q^{6L(6L-1)/2+3} < q^{18L^2}$.

\smallskip
(c) Now consider the case of
$G = \Sp_{2n}(q)$ with $2|q$, or $\Omega_{2n+1}(q)$ with $2 \nmid q$, or $\Omega^+_{2n}(q)$.
By \cite[Lemma 3.4]{GLT2}, $j \leq n+1$, so
$$k+1 \leq (n+6)/3 \leq 2n/3$$
when $n \geq 6$. Applying \cite[Theorem 1.5]{GLT2} we have
$$q^{nL} \geq \chi(1) \geq q^{2k(n-(k+1))} \geq q^{2kn/3},$$
whence $k \leq 3L/2$ and so
\begin{equation}\label{eq-ct3}
  j \leq 3k \leq 9L/2.
\end{equation}
By definition, $\cl(\chi)= j$ means that $\chi$ is an irreducible constituent of $\bigl((\tau+\zeta)|_G\bigr)^j$, where
$\tau$ and $\zeta$ are the reducible  Weil characters of $\SL_{D}(q)$ and $\SU_D(q)$, with $D=2n$ or $2n+1$,
the dimension of the natural module $V:=\F_q^D$ of $G$. In particular, $\tau_G$ contains $1_G$.
Hence, \eqref{eq-ct3} implies that
$N(L) \leq N^*(l)$ with $l:= \lfloor 9L/2 \rfloor$, where $N^*(l)$ is the number of distinct irreducible constituents of
$\bigl(\tau|_G+\zeta|_G)^l$. Since $(\tau|_G)^2=(\zeta|_G)^2$, any irreducible constituent of $\bigl(\tau|_G+\zeta|_G)^l$ is an irreducible constituent of  $(\tau|_G)^l$ or of $(\tau|_G)^{l-1}(\zeta|_G)$. Now, for $0 \leq i \leq l$ we have
$$\langle (\tau|_G)^{l-i}(\zeta|_G)^i, (\tau|_G)^{l-i}(\zeta|_G)^i\rangle =
   \langle (\tau|_G)^{2l-2i}(\zeta|_G)^{2i},1_G \rangle = \langle (\tau^{2l})|_G,1_G \rangle,$$
the number of $G$-orbits on $V^{2l}$. By \cite[Lemma 2.6]{GLT2}, the number of $\tilde G$-orbits on $V^{2l}$
is at most $6q^{l(2l+1)}$, for $\tilde{G} = \Sp_D(q)$ or $\mathrm{GO}_D(q)$, respectively. Note that $[\tilde{G}:G]$ is
$1$ in the $\Sp$-case, $2$ in the $\Omega$-case with $2 \nmid D$ or with $2|q$, and $4$ otherwise.
It follows that
$$N^*(l) \leq 12q^{l(2l+1)/2}\cdot[\tilde{G}:G] < q^{l(2l+1)+5} < q^{50L^2}.$$

\smallskip
(d) Finally, assume that $G = \Omega^-_{2n}(q)$. In this case, $j \leq  n+1$, so
$$k+1 \leq (n+3)/3 \leq 3(n-1)/5$$
when $n \geq 6$. Applying \cite[Theorem 1.5]{GLT2} we have
$$q^{(n-1)L} \geq \chi(1) \geq q^{2k(n-1-k)} \geq q^{4k(n-1)/5},$$
whence $k \leq 5L/4$ and so $j \leq 3k \leq 15L/4$. Now we can repeat the arguments in (c) to get
$N(l) \leq N^*(l) \leq q^{l(2l+1)+5} < q^{41L^2}$.
\end{proof}

We can now prove Theorem~\ref{prod} and Corollary \ref{prod-k}.

\begin{proof}[Proof of Theorem~\ref{prod}]
Let $\epsilon>0$ be given.  We claim there exists $\delta > 0$ such that
if $G$ is a finite quasisimple group of Lie type, and  $\chi_1$ and $\chi_2$ are (possibly reduced) characters of $G$
with $|\chi_1|, |\chi_2| \leq |G|^\delta$, then
$$|\chi_1\chi_2| \geq \bigl(|\chi_1|\cdot|\chi_2|\bigr)^{1-\epsilon}.$$

It suffices to prove the statement in the case $|\chi_1|, |\chi_2| > 1$; in particular, the $\chi_i$ contain nontrivial irreducible constituents
$\alpha_1$ and $\alpha_2$, respectively. We may also assume that each $\chi_i$ is multiplicity-free.
If $G$ is an exceptional group of Lie type, or a finite classical group of rank $\leq 7$,
then \cite{LaSe} implies that $\alpha_i(1) > |G|^{1/20}$. Taking $\delta \leq 1/10$, we may assume $G$ is a classical group of rank
$\geq 8$.  By \cite[Proposition~6.3]{GLT2}, if $G$ is of orthogonal type, the bounds on $|\chi_i|$ imply that they arise from characters
of $\Omega^+_{2n+1}(q)$, $\Omega^+_{2n}(q)$, or $\Omega^-_{2n}(q)$ by composition with a central quotient map.

We will view $\chi_1,\chi_2$ as characters of $\hat{G} = \SL^{\pm}_n(q)$, $\Sp_{2n}(q)$, $\Omega^+_{2n+1}(q)$,
$\Omega^+_{2n}(q)$, or $\Omega^-_{2n}(q)$, respectively, where $n \geq 8$. In all cases, $\hat{G}$ contains a subgroup
$H \cong \SL^{\pm}_m(q)$ with $m:=n$ in the first case, and $H \cong \SL_m(q)$ (inside a split Levi subgroup)
with $m:=n$, $n$, $n$, or $n-1$, respectively, in
the other four cases. One checks that  $q^{m^2/2} < |G| < q^{4m^2}$ in all cases; furthermore, $\alpha_i(1) > q^{m/2}$.
%, in fact, $\chi(1) > q^{m}$ unless $\hat{G}=\SL^\pm_n(q)$ or $\Sp_{2n}(q)$ with $2 \nmid q$, see \cite[Theorem 1.1]{TZ}.
Writing
\begin{equation}\label{eq-pr1}
  q^{mL_i/2} \leq \chi_i(1)=q^{mD_i} \leq q^{mL_i}
\end{equation}
for some (real) numbers $L_i \geq 1$ and $L_i/2 \leq D_i \leq L_i$, we will choose $\delta$ so that
$$q^{mL_i} \leq |\chi_i| \leq |G|^\delta \leq q^{4m^2\delta};$$
in particular,
\begin{equation}\label{eq-pr2}
  L_i \leq 4m\delta.
\end{equation}
Writing
$$\chi_1\chi_2 = \sum_{\gamma \in \Irr(G)}m_\gamma \gamma,~\chi^2_1 = \sum_{\gamma \in \Irr(G)}a_\gamma \gamma,
   ~\chi^2_2 = \sum_{\gamma \in \Irr(G)}b_\gamma \gamma,$$
we have by the Cauchy-Schwarz inequality
$$\begin{aligned}
    \bigl(\sum_\gamma m_\gamma^2\bigr)^2 & = \langle \chi_1\chi_2,\chi_1\chi_2 \rangle^2\\
    & = \langle \chi_1\bar\chi_1,\chi_2\bar\chi_2\rangle^2\\
    & \leq \langle \chi_1\bar\chi_1,\chi_1\bar\chi_1 \rangle \cdot \langle \chi_2\bar\chi_2,\chi_2\bar\chi_2\rangle \\
    & = \langle \chi_1^2,\chi_1^2 \rangle \cdot \langle \chi_2^2,\chi_2^2\rangle \\
    & = \sum_\gamma a_\gamma^2 \cdot \sum_\gamma b_\gamma^2\\
    & \leq \bigl( \sum_\gamma a_\gamma\bigr)^2 \cdot \bigl(\sum_\gamma b_\gamma \bigr)^2,
    \end{aligned}$$
and thus
\begin{equation}\label{eq-pr3}
  \sum_\gamma m_\gamma^2 \leq \sum_\gamma a_\gamma \cdot \sum_\gamma b_\gamma = \bsig(\chi_1^2,G) \cdot \bsig(\chi_2^2,G),
\end{equation}
where, $\bsig(\cdot,\cdot)$, as defined in \cite[(2.1)]{GLT2}, denotes the sum of the multiplicities of all the irreducible constituents of the first argument, regarded as a character on the group specified in the second argument.

First we consider the case $\hat{G}=\SL^\pm_n(q)$.
Then each irreducible constituent $\theta$ of $\chi_i$ lies under some $\tilde\theta \in \Irr(\GL^\pm_n(q))$ of degree
less than $(q\mp 1)\theta(1)$. Hence $\chi_i$ lies under some character $\tilde\chi_i$ of $\GL^\pm_n(q)$ of degree
at most $q^{mL_i+2} \leq q^{5mL_i/4}$ since $m \geq 8$. Applying \cite[Corollary 5.2(i)]{GLT2} to $\tilde\chi_i$ and restricting further down
to $\SL^\pm_n(q)$, we have
$$\bsig(\chi_i^2,G) \leq q^{(125/4)L_i^2+2}\bsig(\chi_i,G)^2 < q^{46L^2}.$$
Here, we use the fact that $\chi_i$ is multiplicity-free, and so $\bsig(\chi_i,G)$ is at most the number of irreducible characters of
$G$ of degree $\leq q^{nL_i}$, which is at most $q^{12L_i^2}$ by Proposition \ref{count}(a).

Next suppose we are in the symplectic-orthogonal case. Taking $\delta > 0$ small enough, we may assume that
$L_i \leq n/9$. Applying Theorem 5.8 and Corollary 5.9 of \cite{GLT2} to $\chi$ and restricting further down
from a normal subgroup of $\GL_m(q)$ containing $H=\SL_m(q)$ to $H$ if necessary, we have
$$\bsig(\chi_i^2,G) \leq q^{A\sqrt{mL_i^3}+60L_i^2}\bsig(\chi_i,G)^2 \leq q^{A\sqrt{mL_i^3}+110L_i^2}$$
for some explicit $A$ (which can be taken to be $70$).
Here, we again use the fact that $\chi_i$ is multiplicity-free, and so $\bsig(\chi_i,G)$ is at most the number of irreducible characters of $G$ of degree $\leq q^{mL_i}$, which is at most $q^{50L_i^2}$ by Proposition \ref{count}(b)--(d).

Using \eqref{eq-pr3}, in either case we have
$$\sum_\gamma m_\gamma^2 \leq q^{A\sqrt{mL_1^3}+A\sqrt{mL_2^3}+BL_1^2+BL_2^2}$$
with $B=110$. It follows that
$$|\chi_1\chi_2| = \sum_{m_\gamma > 0}\gamma(1)^2 \geq \frac{(\sum_\gamma m_\gamma \gamma(1))^2}{\sum_{\gamma}m_\gamma^2}
   = \frac{\chi_1(1)^2\chi_2(1)^2}{\sum_{\gamma}m_\gamma^2}
   \geq q^{2m(D_1+D_2)-A\sqrt{mL_1^3}-A\sqrt{mL_2^3}-BL_1^2-BL_2^2}.$$
Recalling \eqref{eq-pr2}, we observe
$$\frac{\sqrt{mL_1^3}+\sqrt{mL_2^3}}{m(D_1+D_2)} \leq \frac{2mL_1\sqrt{\delta}+2mL_2\sqrt{\delta}}{m(L_1+L_2)/2} = 4\sqrt{\delta},$$
and
$$\frac{L_1^2+L_2^2}{m(D_1+D_2)} \leq \frac{2mL_1\delta+2mL_2\delta}{m(L_1)+L_2)/2} = 4\delta.$$
Hence
$$|\chi_1\chi_2| \geq q^{2m(D_1+D_2)(1-2A\sqrt{\delta}-2B\delta)}.$$
Note that $|\chi_1| \cdot |\chi_2| \leq \chi_1(1)^2\chi_2(1)^2 = q^{2m(D_1+D_2)}$. Hence,
taking $\delta > 0$ so that
$$2A\sqrt{\e}+2B\e \leq \epsilon,$$
we have
$|\chi_1\chi_2| \geq \bigl(|\chi_1|\cdot|\chi_2|\bigr)^{1-\epsilon}$, as desired.
\end{proof}

\begin{proof}[Proof of Corollary \ref{prod-k}]
We prove by induction on $k \geq 2$ the equivalent statement:\\
{\it For any $\epsilon > 0$ and any $k \geq 2$, there exists an explicit $\gamma > 0$ (depending on both $\epsilon$ and $k$) such that the following statement holds. If $G$ is a finite quasisimple group
of Lie type and $\chi_1,\chi_2, \ldots,\chi_k$ are any characters of $G$ with
$|\chi_1|, |\chi_2|, \ldots, |\chi_k| \leq |G|^\gamma$, then
$$|\chi_1\chi_2 \cdots \chi_k| \geq \bigl(|\chi_1|\cdot|\chi_2|\cdots |\chi_k|\bigr)^{1-k\epsilon}.$$}
We will show that this statement holds with $\gamma:=\delta/(k-1)$, where $\delta$ the constant in
Theorem \ref{prod}. The case $k=2$ already established in Theorem \ref{prod}. For the inductive step, note by Lemma \ref{sub-add-mult}
and the induction hypothesis that
$$\bigl(|\chi_2| \cdots |\chi_k|\bigr)^{1-(k-1)\epsilon} \leq |\chi_2 \ldots \chi_k| \leq \prod^k_{i=2}|\chi_i| \leq |G|^{\gamma(k-1)} \leq |G|^\delta.$$ Since $|\chi_1| \leq |G|^\delta$,
by Theorem \ref{prod} we have
$$|\chi_1\chi_2 \ldots \chi_k| \geq \bigl( |\chi_1|\cdot |\chi_2 \ldots \chi_k|\bigr)^{1-\epsilon} \geq
   \bigl( |\chi_1|\cdot (|\chi_2| \cdots |\chi_k|)^{1-(k-1)\epsilon}\bigr)^{1-\epsilon} \geq \bigl(|\chi_1|\cdot|\chi_2| \cdots |\chi_k|\bigr)^{1-k\epsilon}.$$
\end{proof}

The next result is required in our proof of Theorem~\ref{char-growth}.

\begin{theor}\label{product}
For any $\delta > 0$, there exists an explicit integer $N$ such that the following statement holds. If $G$ is a finite simple group
of Lie type, and $\chi_1, \ldots,\chi_N$ are any (not necessarily irreducible) characters of $G$ with $|\chi_i| \geq |G|^\e$ for
all $i$, then $|\chi_1\chi_2 \cdots \chi_N| = |G|$ and thus $\chi_1 \chi_2 \cdots \chi_N$ contains every irreducible character of $G$.
\end{theor}

\begin{proof}
For any character $\chi$ of $G$, let $\chi^*$ denote (some) irreducible constituent of largest degree of $\chi$.

First we consider the case $k(G) \geq |G|^{\delta/3}$. If $G$ is of rank $r$ over $\F_q$, then
by \cite{FG} we have $k(G) \leq (27.2)q^r$, whereas $|G| \geq q^{r^2}$. It follows that $r \leq r_0$ is bounded (in terms of $\e$).
On the other hand, $|\chi_i| > 1$ implies that $\chi^*_i \neq 1_G$, and so $\chi^*_i(1) > q^r/3$ by \cite{LaSe}. As $r \leq r_0$,
$\chi^*_i(1) \geq |G|^{\delta_0}$ for some $\delta_0$ depending on $\e$. Applying \cite[Theorem 8.5]{LT}
and taking $N \geq \epsilon/\delta_0$ (with $\epsilon$ the constant in \cite[Theorem 8.5]{LT}),
we see that $\chi^*_1 \ldots \chi^*_N$ contains $\Irr(G)$.

Now suppose that $k(G) \leq |G|^{\delta/3}$. Then $\chi^*_i(1) \geq (|\chi_i|/k(G))^{1/2} \geq |G|^{\delta/3}$.
Applying \cite[Theorem 8.5]{LT}
and taking $N \geq 3 \epsilon/\e$, we again see that $\chi^*_1 \ldots \chi^*_N$ contains $\Irr(G)$.
\end{proof}

\begin{proof}[Proof of Theorem \ref{char-growth}]
Fix $\epsilon > 0$. It follows from Theorem \ref{prod}, applied for $\chi_1=\chi_2 = \chi$ with $\epsilon$ replaced by $\epsilon/2$,
that there exists some $\e > 0$ depending only on $\epsilon$ such that if $|\chi| \leq |G|^{\e}$, then
$|\chi^2| \geq |\chi|^{2-\epsilon}$. By Theorem \ref{product},  there is some integer $N > 0$ depending only on $\epsilon$, such that
if $|\chi| \geq |G|^{\delta}$ then $\chi^N$ contains $\Irr(G)$. The result follows.
\end{proof}

\begin{proof}[Proof of Theorem \ref{cover}]
For any character $\chi$ of $G$, again let $\chi^*$ denote (some) irreducible constituent of largest degree of $\chi$.
It suffices to prove the statement in the case $|\chi_i| > 1$, and thus $\chi^*_i \neq 1_G$. By the results of \cite{FG} and
\cite{LaSe}, the degree of any non-trivial irreducible character of $G$ is
at least $k(G)^{1/6}$; in particular,
$$\chi^*_i(1) \geq k(G)^{1/10} \geq m_i^{1/6},$$
if $m_i$ denotes the number of distinct irreducible constituents of $\chi_i$. It follows that
$$|\chi_i| \leq m_i\chi^*_i(1)^2 \leq \chi^*_i(1)^{8},$$
and so $\chi^*_i(1) \geq |\chi_i|^{1/8}$. Now applying \cite[Theorem 8.5]{LT} and taking $c=8\epsilon$ (with $\epsilon$ the constant in \cite[Theorem 8.5]{LT}), we have that $\prod^m_{i=1}\chi^*_i(1) \geq |G|^\epsilon$, and so $\chi^*_1 \ldots \chi^*_N$ contains $\Irr(G)$.
\end{proof}

\section{Proof of Theorem~\ref{N=6}}

First we prove the theorem when $G = \PSL_2(q)$.  Let $\chi_1$ and $\chi_2$ denote any irreducible characters of $G$.
Then $|\chi_i(g)| \le q^{1/2}$ for all $g\in G$, $|\chi_i(g)| \le 2$ for all non-trivial semisimple $g\in G$, and the total number of non-semisimple elements of $G$ is less than $q$.
Therefore
$$\sum_{g\neq 1} |\chi_1^6(g)\chi_2(g)| \le \chi_2(1) q^4 + 128q^3.$$
If $\chi_1$ is non-trivial, it has degree at least $\frac{q-1}2$, so for $q$ sufficiently large,
$$\sum_{g\in G} \chi_1^6(g)\chi_2(g) \neq 0,$$
and applying this to any non-trivial constituent $\chi_1$ or $\chi_2$, we get the desired result.

\smallskip
As before, we use the notation $\PSL^\epsilon$ to denote $\PSL$ when $\epsilon =+$ and $\PSU$ when $\epsilon=-$.
Next we consider the case $G=\PSL^\epsilon_3(q)$.  Again, the generic character table has been computed explicitly \cite{SF}.  There are $O(q)$ characters of degree $\le q^2+q+1$,
and all other characters have degree at least $q^3/4$ (when $q$ is not too small).  Therefore, the sum of the squares of the degrees of irreducible characters of $G$ of degree $< q^3/4$
is smaller than $|G|^{11/12}$ for large values of $q$.  We therefore assume $\chi$ has an irreducible constituent $\chi_1$ of degree $\ge q^3/4$.
We have $|\chi_i(g)| = O(q)$ for all irreducible characters and all non-trivial elements, so
$$\sum_{g\neq 1} |\chi_1^6(g) \bar\theta(g)| = O(q^6 |G|\theta(1)) = O(q^{14}\theta(1)) = o(\chi_1(1)^6\theta(1)).$$
Therefore, if $\chi_1(q) \ge q^3/4$ and $q$ is sufficiently large, any $\theta \in \Irr(G)$ is a constituent of $\chi_1^6$.

\smallskip
We may therefore assume $G = \PSL^\epsilon_n(q)$ with $n\ge 4$.  We claim that for $q$ sufficiently large, $|\chi| \ge |G|^{11/12}$ implies $\chi$ has an irreducible constituent $\chi_1$ with
\begin{equation}
\label{44}
\chi_1(1) \ge q^{\frac{44n^2}{131}}.
\end{equation}

Indeed, the number of characters of $G$ which do not satisfy this inequality is less than $k(G) = O(q^{n-1})$ and
$|G|=O(q^{n^2-1})$, so it suffices to note that
$$\frac{88n^2}{131} + (n-1) < \frac{11(n^2-1)}{12}$$
for $n\ge 4$.

\smallskip
We claim that if $q$ sufficiently large in terms of $n$, \eqref{44} implies that  the Steinberg character $\St$  is a constituent of $\chi_1^3$.
If $\epsilon=+$ or if $\epsilon=-$ but $2|n$, then every irreducible character of $G$ appears in $\St^2$ \cite[Theorem~1.2]{HSTZ}, and so it follows that $|\chi_1^6| = |G|$.
Consider the case $\epsilon=-$ and $2 \nmid n$. Then by \cite[Theorem~1.2]{HSTZ},
$\St^2$ contains all but one irreducible character $\alpha$ of $\PSU_n(q)$,
which has smallest degree $(q^n-q)/(q+1)$ among all nontrivial irreducible characters of $G$. Let $\theta \in \Irr(G)$. By Proposition~\ref{lem:bound1}(i) and \eqref{44},
$$|\theta\bar\chi_1| \geq |\chi_1(1)| > |\alpha|,$$
and so $\theta\bar\chi_1 \neq m\alpha$ for any $m \in \N$ by \eqref{44}. Hence $\theta\bar\chi_1$ contains some
irreducible character $\beta \neq \alpha$, whence $\langle \St^2,\beta \rangle > 0$ and so
$$\langle \chi_1^7,\theta \rangle = \langle \chi_1^6,\theta\bar\chi_1\rangle \geq \langle \St^2,\beta \rangle > 0,$$
and thus $\chi_1^7$ contains $\Irr(G)$, as desired.

\smallskip
It suffices to show that for large enough $q$,
\begin{equation}
\label{cube contains Steinberg}
\sum_{g\neq 1} \frac{|\St(g)\chi_1(g)^3|}{\St(1) \chi_1(1)^3} < 1.
\end{equation}

As $\St$ vanishes at all non-semisimple elements, we need only consider semisimple elements $g$ in \eqref{cube contains Steinberg}.
Let $X_s$ denote the set of conjugacy classes of semisimple elements in $G$ of support $s$, i.e., for which the largest dimension of any $\bar\F_q$-eigenspace is $n-s$.
We can rewrite \eqref{cube contains Steinberg} as
\begin{equation}
\label{sum over supports}
\sum_{s=1}^{n-1} \sum_{C\in X_s} |C|\frac{|\St(C)|}{\St(1)} \biggl(\frac{|\chi_1(C)|}{\chi_1(1)}\biggr)^3 < 1.
\end{equation}

We consider the factors of the summand in the left hand side of \eqref{sum over supports}, one by one.
If $C\in X_s$ has an element $g$ represented by a matrix $M$ with eigenvalue multiplicities $m_1,\ldots,m_k$, then the centralizer of $M$ in $\GL^\epsilon_n(q)$ is the group of $\F_q$ points (possibly twisted) of a connected algebraic group of dimension $\sum_i m_i^2$, so its order is $(1+o(1))q^{\sum_i m_i^2}$.  Therefore, $C$, which is contained in the $\PSL^\epsilon_n(q)$-conjugacy class of $g$,
%which admits a surjective map from the $\GL_n(\F_q)$-conjugacy class of $M$,
has cardinality
$$O(\gcd(n,q-\epsilon)q^{n^2 - \sum_i m_i^2})=O(q^{n^2 - \sum_i m_i^2}).$$

The absolute value of $\St(C) = \St(g)$ is the order of a $p$-Sylow subgroup of the centralizer of $M$, which is $q^{\sum_i \binom{m_i}2}$,
and $\St(1) = q^{\binom{n}{2}}$.
Therefore,
$$\frac{|\St(C)|}{\St(1)} = q^{-\frac{n^2 - \sum_i m_i^2}2},$$
and
$$|C| \frac{|\St(C)|}{\St(1)} = O\Bigl(q^{\frac{n^2 - \sum_i m_i^2}2}\Bigr).$$
As some $m_i$ equals $n-s$, we have
$$\sum_i m_i^2 \ge (n-s)^2 + s\cdot 1^2,$$
so
$$|C| \frac{|\St(C)|}{\St(1)} =  O\Bigl(q^{\frac{2sn-s^2-s}2}\Bigr).$$
By \cite[Theorem 1.13]{TT} and \cite[Theorem 3.1]{LiST}, for $q$ large enough we have
$$\frac{|\chi_1(g)|}{\chi_1(1)} = O(\chi_1(1)^{-s/n}) = O(q^{-\frac{44sn}{131}}),$$
where the implicit constant depends only on $n$.  Thus, the summand as a whole is $O(q^{-\frac{sn}{131}-\frac{s^2+s}2})$.

\smallskip
Finally, we claim that $|X_s| = O(q^s)$.  Since each conjugacy class of $\PGL^\epsilon_n(q)$ decomposes into at most $\gcd(n,q-\epsilon)$ conjugacy classes of $G$, it suffices to prove the same thing for $\PGL^\epsilon_n(q)$ conjugacy classes containing semisimple elements of $\PSL^\epsilon_n(q)$ or, indeed, $\GL^\epsilon_n(\F_q)$ conjugacy classes containing semisimple elements of $\SL^\epsilon_n(\F_q)$. Since any such element has connected centralizer
in $\GL_n(\overline{\F_q})$, any such class is uniquely determined by the spectrum of its representatives, equivalently, by their
common characteristic polynomial.

Let $Y_{s,i}$ denote the set of conjugacy classes in $\GL^\epsilon_n(\F_q)$ of semisimple elements $M$ in $\SL^\epsilon_n(\F_q)$ which have exactly $i$ eigenvalues of multiplicity $n-s$ (and all other eigenvalues of lower multiplicity), where
$1 \leq i  \leq t:= \lfloor n/(n-s) \rfloor$.
Let $P(x)$ be the monic polynomial of degree $i$ whose roots are the $i$ eigenvalues of $M$ of multiplicity $n-s$, and $Q(x)$ be the monic polynomial of degree $n-i(n-s)$, such that $P(x)^iQ(x)$ is the characteristic polynomial of $M$. Then
$$\det P(0)^iQ(0) = (-1)^n\det(M) = (-1)^n.$$
So the constant term of $P(x)$ is determined up to at most $i \leq n$ possibilities by the constant term of $Q(x)$, and so there are at most $nq^{i-1}$
possibilities for $P(x)$ for any fixed (constant term of) $Q(x)$.
Since there are at most $q^{n-i(n-s)}$ possibilities for $Q(x)$, we see that the total number of possibilities for $P(x)^iQ(x)$,
which, as explained above, gives an upper bound for $|Y_{s,i}|$, is at most
$$n q^{i-1} q^{n-i(n-s)} = nq^{s+(1-i)(n-s-1)}=O(q^s).$$
Since $q$ is large compared to $n$, it follows that $|X_s| = \sum^t_{i=1}|Y_{s,i}| = O(q^s)$, as claimed.

We conclude that the left hand side of \eqref{sum over supports} is $O(q^{-sn/131})$.  Taking $q$ large enough, the statement now follows.

\section{Semisimple Compact Lie groups}

In this section, we briefly consider the situation when $G$ is a compact (connected) semisimple Lie group instead of a finite simple group of Lie type.
For $\chi$ a character of $G$,  the definition of $|\chi|$ works as before.
We have the following analogue of Proposition~\ref{lem:bound1}, which no longer requires irreducibility.

\begin{prop}
\label{compact-growth}
Let $G$ be a semisimple compact Lie group, $n\ge 2$, and $\chi_1,\chi_2,\ldots,\chi_n$ non-trivial characters.  Then $|\chi_1\chi_2\cdots\chi_n| > \max(|\chi_1|,|\chi_2|,\ldots,|\chi_n|)$.
\end{prop}

\begin{proof}
It suffices to treat the case $n=2$ and to prove in this case that if $\chi_1$ is non-trivial, then $|\chi_1 \chi_2| > |\chi_2|$.   It is enough to treat the case that $\chi_1$ is irreducible.
Let $\lambda_1$ denote the highest weight of $\chi_1$.  Let $\{\varphi_1,\ldots,\varphi_k\}$ denote the irreducible constituents of $\chi_2$, and let $\mu_i$ denote the highest weight of $\varphi_i$.

Then $\lambda_1 + \mu_i$ is the highest weight in $\chi_1 \varphi_i$, so there is an irreducible constituent $\psi_i$ of $\chi_1 \chi_2$ with highest weight $\lambda_1 + \mu_i$.
By the Weyl dimension formula,
\begin{equation}
\label{Weyl squared}
|\chi_1 \chi_2| \ge \sum_{i=1}^k |\psi_i|^2 = \sum_{i=1}^k \prod_{\alpha\succ 0} \frac{(\delta+\lambda_1+\mu_i,\alpha)^2}{(\delta,\alpha)^2}
= \sum_{i=1}^k \prod_{\alpha\succ 0} \frac{((\delta+\mu_i,\alpha)+(\lambda_1,\alpha))^2}{(\delta,\alpha)^2},
\end{equation}
where $\delta$ denotes half the sum of the positive roots.
As $\lambda_1$ is a non-zero dominant weight, we have $(\lambda_1,\alpha)$ non-negative for all positive roots $\alpha$ and strictly positive for at least one of them.
Therefore, the right hand side of \eqref{Weyl squared} is strictly greater than
$$\sum_{i=1}^k \prod_{\alpha\succ 0} \frac{(\delta+\mu_i,\alpha)^2}{(\delta,\alpha)^2} = \sum_{i=1}^k |\varphi_i|^2 = |\chi_2|.$$

\end{proof}

When $G$ is of positive dimension, it no longer makes sense to compare $|\chi|$ to $|G|$, so we do not have
an analogue of Theorem~\ref{prod} or any of the subsequent results proved above for groups of Lie type.
We can still ask about power growth of $|\chi|$.
The case of $\SU(2)$ illustrates the situation:
we get uniform power growth when $\chi$ is irreducible, but there is no such growth for general characters.

\begin{prop}
Let $G=\SU(2)$.
\begin{enumerate}
\item[\rm(i)]
If $\chi$ ranges over irreducible characters of $G$, then
$$\lim_\chi \frac{\log |\chi^2|}{\log |\chi|} = \frac 32.$$
\item[\rm(ii)]
If $\chi$ ranges over all characters of $G$, then
$$\liminf_\chi \frac{\log |\chi^2|}{\log |\chi|} = 1.$$
\end{enumerate}
\end{prop}

\begin{proof}
If $\chi_n$ denotes the unique  irreducible character of $G$ of degree $n+1$, then by the Clebsch-Gordan formula (\cite[\S22,~Ex.~7]{Hump}),
$$\chi_n^2 = \sum_{i=0}^n \chi_{2i}.$$
Thus,
$$|\chi_n^2| = \sum_{i=0}^n (2i+1)^2 = \binom{2n+3}{3},$$
while $|\chi_n| = (n+1)^2$.  This gives (i).

Again by Clebsch-Gordan, $\chi_n^4$ is a linear combination of $\chi_0,\chi_2,\chi_4,\ldots,\chi_{4n}$ with all coefficients positive, so
$$|\chi_n^4| = \binom{4n+3}{3}.$$
Therefore, as $\chi$ ranges over $\{\chi_n^2,\mid n\in \N\}$, the limit of $\frac{|\chi^2|}{|\chi|}$ is $8$, and the limit of $\frac{\log|\chi^2|}{\log|\chi|}$ is $1$.
\end{proof}

Next, we prove Theorem~\ref{compact} establishing uniform power growth for irreducible characters of any fixed compact semisimple Lie group.

%\begin{theor}
%Let $G$ be a compact semisimple Lie group.  Then there exists $\epsilon > 0$ such that for each irreducible character $\chi$ of $G$, we have %$|\chi^2| \ge |\chi|^{1+\epsilon}$.
%g\end{theor}

\begin{proof}[Proof of Theorem \ref{compact}]
Let $\varpi_1,\ldots,\varpi_r$ denote the fundamental weights of $G$.
Let $\chi_\lambda$ be the irreducible character with highest weight $\lambda = a_1 \varpi_1+\cdots+a_r\varpi_r$.
We define
$$\langle \lambda,\alpha\rangle := \frac{2(\lambda,\alpha)}{(\alpha,\alpha)}.$$
By \cite[\S22,~Ex.~1]{Hump}, $\lambda-k\alpha_i$ is a weight of $\chi_\lambda$ for $0\le k\le \langle \lambda,\alpha_i\rangle = a_i$.

We claim that $\delta+2\lambda - k\alpha_i$ is dominant for $0\le k\le a_i$.  It suffices to check the non-negativity of $\langle \delta+2\lambda-k\alpha_i,\alpha_j\rangle$
for $1\le j\le r$.  For $j\neq i$, we have
$$\langle \delta+2\lambda-k\alpha_i,\alpha_j\rangle = 1+2a_j - \langle \alpha_i,\alpha_j\rangle \ge 1,$$
and for $j=i$, we have
$$\langle \delta+2\lambda-k\alpha_i,\alpha_i\rangle = 1+2a_i - 2k \ge 1.$$
Therefore, by a theorem of Brauer \cite[\S24,~Ex.~9]{Hump},
for each $k$ in this range, $\chi_\lambda^2$ contains the irreducible character with highest weight $2\lambda-k\alpha_i$.

By the Weyl dimension formula \cite[\S24.3]{Hump},
$$\chi_\lambda(1) = \prod_{\alpha\succ 0}\frac{(\lambda+\delta,\alpha)}{(\delta,\alpha)},$$
where the product is taken over the set $\Phi^+$ of positive roots $\alpha$.  Now, for $1\le i\le r$, $(\varpi_i,\alpha)\ge 0$,
so regarded as a function in $\lambda$, $\chi_\lambda(1)$ is a product of affine linear functions in $\lambda$ which take positive values
in the cone of dominant weights.  Thus, if $a_i\ge 1$ and $0\le k\le a_i$, then
\begin{equation}
\label{big-factors}
\chi_{2\lambda-k\alpha_i}(1) \ge (1-k/a_i)^{|\Phi^+|} \chi_{2\lambda}(1)\ge (1-k/a_i)^{|\Phi^+|}\chi_\lambda(1).
\end{equation}

We fix $i$ such that $a_i = \max(a_1,\ldots,a_r)$.  Thus,
$$|\chi_\lambda| = \chi_\lambda(1)^2 = O(a_i^{2|\Phi^+|}) = O(a_i^{|\Phi|}).$$
On the other hand, by \eqref{big-factors},
$$|\chi_\lambda^2| \ge \sum_{k=0}^{a_i} |\chi_{2\lambda-k\alpha_i}| \ge |\chi_\lambda| \sum_{k=0}^{a_i} (1-k/a_i)^{|\Phi|}.$$
As $a_i\to \infty$, the sum $\sum_k (1-k/a_i)^{|\Phi|}$ can be bounded below by a positive constant multiple of $a_i$ and therefore by a positive constant
multiple of $|\chi_\lambda|^{1/|\Phi|}$.  If $\epsilon < 1/|\Phi|$, therefore, $|\chi_\lambda^2| \ge |\chi_\lambda|^{1+\epsilon}$ with finitely many exceptions $\lambda$.
The theorem is trivial for $\lambda=0$, and for each $\lambda\neq 0$, it holds when $\epsilon > 0$ is small enough by Proposition~\ref{compact-growth}.

\end{proof}

Note that, unlike Theorem~\ref{main irred}, Theorem~\ref{compact} does not guarantee that power growth is uniform in $G$.  This is unavoidable, as the following proposition shows.

\begin{prop}
There exists a sequence $\chi_2,\chi_3,\ldots$ of non-trivial irreducible characters of the Lie groups $\SU(2),\SU(3),\ldots$ respectively, such that
$$\liminf_n \frac{\log |\chi_n^2|}{\log |\chi_n|} = 1.$$
\end{prop}

\begin{proof}
We choose for each $n\ge 2$ a positive integer $k_n$ and define $\chi_n$ to be the character of the irreducible representation of $\SU(n)$ with highest weight
$\lambda_n:=k_n\delta_n$,
where $\delta_n$ is half the sum of the set $\Phi_n^+$ of positive roots of $\SU(n)$.
By the Weyl dimension formula,
$$\chi_n(1) = \prod_{\alpha\in \Phi_n^+} \frac{(\lambda_n+\delta_n,\alpha)}{(\delta_n,\alpha)} = (k_n+1)^{|\Phi_n^+|} = (k_n+1)^{n(n-1)/2}.$$
On the other hand, for each fixed $n$, the weights of $\chi_n$ are contained in the convex hull of $\{w(\lambda_n)\mid w\in \SSS_n\}$.
This can be expressed as $k_n X_n$, where $X_n$ denotes the convex hull of $\{w(\delta_n)\mid w\in \SSS_n\}$.
Thus, the weights of $\chi_n^2$ are contained in $2k_n X_n$.
Since the highest weight of every irreducible constituent of $\chi_n^2$ is a lattice point in the fixed polytope $2X_n\subset \R^{n-1}$, scaled by $k_n$, the number of such highest weights is  $O(k_n^{n-1})$, where the implicit constant depends only on $n$.

On the other hand, by the Weyl dimension formula, the degree of each irreducible factor of $\chi_n^2$ is $O(k_n^{n(n-1)/2})$.  Therefore, assuming that each $k_n$ is sufficiently large, we can guarantee
$$\frac{\log |\chi_n^2|}{\log |\chi_n|} < \frac{n+1}n,$$
which implies the proposition.
\end{proof}

\end{document}